\documentclass{amsart}
\usepackage{amsmath,amsfonts,amssymb,amsthm,xcolor,amscd}
\usepackage{latexsym,xspace,enumerate}
\usepackage[mathscr]{eucal}
\usepackage[all]{xypic}
\usepackage{hyperref}
\usepackage{vmargin}
\setmarginsrb{15mm}{0mm}{15mm}{5mm}%
          {10mm}{7mm}{10mm}{10mm}

 \newtheorem{thm}{Theorem}[section]
 \newtheorem{cor}[thm]{Corollary}
 \newtheorem{lem}[thm]{Lemma}
 \newtheorem{prop}[thm]{Proposition}
 \theoremstyle{definition}
 
 \theoremstyle{remark}
 \newtheorem{rem}[thm]{Remark}
  
 \newtheorem{ex}[thm]{Example}
 \newtheorem{que}[thm]{Question}
 \numberwithin{equation}{section}

\numberwithin{equation}{section}

\numberwithin{equation}{section}

\begin{document}

\title[Capability of nilpotent Lie algebras of small dimension ]
{ Capability of nilpotent Lie algebras of small dimension}

\author[F. Pazandeh Shanbehbazari]{Fatemeh Pazandeh Shanbehbazari}
\author[P. Niroomand]{Peyman Niroomand}

\address{School of Mathematics and Computer Science\endgraf
University of Damghan \endgraf
Damghan, Iran}
\email{p$\_$niroomand@yahoo.com, fateme.pazandeh@gmail.com, shamsaki.afsaneh@yahoo.com}

\author[F.G. Russo]{Francesco G. Russo}
\address{Department of Mathematics and Applied Mathematics \endgraf
University of Cape Town \endgraf
Private Bag X1, 7701, Rondebosch \endgraf
Cape Town, South Africa}
\email{francescog.russo@yahoo.com}

\author[A. Shamsaki]{Afsaneh Shamsaki}



\date{\today}

\begin{abstract}
Given a  nilpotent Lie algebra $L$ of dimension $\le 6$ on an arbitrary  field of characteristic $\neq 2$, we show a direct method which allows us to detect the capability of $L$ via computations on the size of its nonabelian exterior square $L \wedge L$.  For dimensions higher than $  6$, we show a result of general nature, based on the evidences of the low dimensional case, focusing on  generalized Heisenberg algebras. Indeed  we  detect the capability of $L \wedge L$ via the size of the Schur multiplier $M(L/Z^\wedge(L))$ of $L/Z^\wedge(L)$, where $Z^\wedge(L)$ denotes the exterior center of $L$.
 \end{abstract}

\maketitle

\section{Statement of the results and terminology}
Throughout this paper all Lie algebras are considered over a prescribed field $\mathbb{F}$ and $[ \ , \ ]$ denotes the Lie bracket. Let $L$ and $K$ be two Lie algebras. Following \cite{ellis, ellis2}, an action of $L$ on $K$ is an $\mathbb{F}$-bilinear map $L\times K \rightarrow  K$ given by  $(l,k)\rightarrow \, ^{l}k$ satisfying
$$\, ^{[l,l^{\prime}]}k= \, ^{l}(\, ^{l^{\prime}}k)-  \, ^{l^{\prime}}(\, ^{l}k) \  \mbox{and} \, ^{l}[k,k^{\prime}]=[\, ^{l}k, k^{\prime}]+[k, \, ^{l}k^{\prime}],$$ for all $l, l^{\prime} \in L$ and $k, k^{\prime}\in K$.
These actions are  \textit{compatible}, if $\, ^{\, ^{k}l}k^{\prime} = \, ^{^{k^{\prime}}l}k$ and  $\, ^{\, ^{l}k}l^{\prime} = \, ^{\, ^{l^{\prime}}k}l$ for all  $l, l^{\prime} \in L$ and $k, k^{\prime}\in K$. 
 Now $L\otimes K$ is the Lie algebra generated by the symbols $l \otimes k$ subject to the relations $$c(l \otimes k)=cl \otimes k=l \otimes ck, \ (l+l^{\prime}) \otimes k =l \otimes k + l^{\prime} \otimes k,$$ $$ l \otimes (k+k^{\prime})= l \otimes k + l \otimes k^{\prime},  \ [l,l^{\prime}]\otimes k=l \otimes \, ^{l^{\prime}}k - l^{\prime} \otimes \, ^{l}k,$$ $$l \otimes [k, k^{\prime}] = \, ^{k^{\prime}}l \otimes k - \, ^{k}l \otimes k^{\prime}, \  [l \otimes k , l^{\prime}\otimes k^{\prime}]= -\, ^{k}l \otimes \, ^{l^{\prime}}k^{\prime},$$
where $c\in \mathbb{F}$, $l, l^{\prime} \in L$ and $k, k^{\prime} \in K$.

    In case $L=K$, one can find the \textit{nonabelian tensor square} $L \otimes L$ as special case, and, if in addition $L$ is abelian, we get the well known abelian tensor product $L \otimes L$ of  $L$. The map $$\kappa : l \otimes l^{\prime} \in L \otimes L \mapsto [l, l^{\prime}] \in L^{2}:=[L,L]=\langle [l,l^{\prime}] \ | \ l,l^{\prime} \in L \rangle$$ turns out to be  epimorphism from $L \otimes L$ to the derived subalgebra $L^2$. It is well known (see \cite{ellis}) that both $J_2(L)= \ker{\kappa}$ and the   \textit{ diagonal ideal} $ L \square L=\langle l \otimes l  \ | \  l \in L \rangle$ are central  in $L \otimes L$. The quotient Lie algebra 
 $$L \wedge L= (L \otimes  L)/ (L \square L)=\langle l \otimes l + (L\square L)  \ | \  l \in L \rangle = \langle l \wedge l  \ | \  l \in L \rangle$$
is called  \textit{nonabelian exterior square} of $L$ and it is easy to check that  $$\kappa^{\prime}: l \wedge l^{\prime}  \in L \wedge L \mapsto [l, l^{\prime}]\in  L^{2}$$ is epimorphism as well. Classical results from \cite{ellis2} show that $ \ker \kappa^{\prime} \simeq M(L)$, where $M(L)$ denotes the \textit{Schur multiplier } of $L$. There are various ways to define $M(L)$, and one is to put $M(L)=H_2(L,\mathbb{Z})$, that is, as the second dimensional Lie algebra of homology on $L$ with coefficients in the ring $\mathbb{Z}$ of the integers.

Of course, $L/L^2$ is abelian, so $L / L^2 \otimes L / L^2$, and in particular $L / L^2 \square L / L^2$, is well known, since we have a decomposition in the direct sum of ideals of dimension one for abelian Lie algebras (see \cite{weibel}). In fact if $L$ is of finite dimension $\mathrm{dim} \ L$, then 
$$\frac{L}{L^2} \square \frac{L}{L^2} \simeq \underbrace{A(1) \oplus A(1) \oplus \ldots \oplus A(1)}_{k-\mbox{times}},$$
where $k \le \mathrm{dim} \ L$ and the abelian Lie algebra of dimension $n$ is denoted by $A(n)$. Therefore one could ask whether a similar decomposition is known more generally for $L \square L$ and  this has been recently investigated in \cite[Proposition 3.3, Theorems 3.6 and 3.12 ]{johari3}, where $[L,L]$ is replaced by $[L,N]$ with $N$ suitable ideal of $L$.

Originally, results of splitting of the diagonal ideal $L \square L$ are related to the categorical properties of the \textit{universal quadratic Whitehead functor} $\Gamma $, introduced by Whitehead \cite{white} for groups and by  Ellis \cite{ellis} for Lie algebras. Among the properties of $\Gamma$, one can note the existence of a natural epimorphism of Lie algebras $ \gamma (l +L^2) \in \Gamma(L/L^2) \longmapsto l \otimes l \in L \square L$, which has been recently studied in \cite{johari3} in connection with $L \square L$. In fact the size of $L \square L$ may be related to the property of a Lie algbra of being isomorphic to the central quotient of another Lie algebra, that is, to the notion of capability; i.e.: we say that a Lie algebra $L$ is \textit{capable} if $L \simeq E/Z(E)$ for some other Lie algebra $E$. The \textit{ exterior center} $Z^\wedge(L)$ turns out to be a central ideal of  $L$ and is defined by   $$Z^{\wedge}(L)=\lbrace l \in L \ | \  l \wedge l^{\prime}=0, \forall l^{\prime} \in L\rbrace$$ and  \cite{ellis, ellis2} show that  $L$ is capable if and only if $Z^{\wedge}(L)=0$. A series of contributions focused on the property of being capable via the size of the exterior centre after \cite{ellis, ellis2}.

Our first main result deals with those nilpotent Lie algebras of finite dimension, classified in \cite{serena, degraaf} and reported below in Theorem \ref{classification}. Note that there are no assumptions on the nature of the ground field of the Lie algebra  here.

  \begin{thm}\label{main1}
The only noncapable nilpotent Lie algebras of dimension at most $6$ are  $ L_{5,4}, $ $ L_{6,4}, $ $L_{6,10}$,  $L_{6,14}$, $L_{6,16}$, $L_{6,19}$, $L_{6,20}$ and $A(1)$. 
\end{thm}  

The notations, introduced in Theorem \ref{main1}, are recalled in Section 2 and follow \cite{serena, degraaf}. A computational aspect is discussed in Sections 2 and 3, in order to find the nonabelian exterior square, the nonabelian tensor square and the Schur multipliers of all nilpotent Lie algebras up to dimension six. The main results appear in Section 3.

Our second result is motivated by an evidence which we noted up to the  dimension six. We can basically control some extremal situations but we cannot control the capability of $L \wedge L$ completely via that of $L^2$.

\begin{thm}\label{francesco2}
Assume that $L$ is a nonabelian nilpotent Lie algebra of finite dimension $n \ge 3$ and that $L^2 /  Z^\wedge(L)$ is capable. Then  $Z^\wedge(L \wedge L)$ is isomorphic to a subalgebra of $M(L/Z^\wedge(L))$.  
\end{thm}

In other words, Theorem \ref{francesco2} shows that the capability of $L \wedge L$ can be controlled by that of $L^2/Z^\wedge(L)$ for high dimensions, provided $L^2 /  Z^\wedge(L)$ is capable. Notations and terminology are standard and follow \cite{alamian, dieter, salemkar, N1bis, N2, weibel}.

\section{Preliminaries up to dimension five} 

Some well known properties of the Ganea map are summarised here. This is done just for convenience of the reader.

 \begin{prop}[See Corollary 2.3 in \cite{N3} and Proposition 4.1 in \cite{alamian}]  \label{a5}\
  
     Let $L$ be a Lie algebra and $N$ be a central ideal of $L$. \begin{itemize}
     \item[(i)]$N \subseteq Z^{\wedge}(L)$ if and only if $M(L) \rightarrow M(L/N) $ is a monomorphism,
     \item[(ii)]$N \subseteq Z^{\wedge}(L)$ if and only if  $L \wedge L  \rightarrow L/N \wedge L/N $ is a monomorphism,
    \item[(iii)]$L \wedge N \rightarrow L \wedge L \rightarrow L/N \wedge L/N \rightarrow 0$ is exact,
 \item[(iv)]$ M(L) \rightarrow M(L/N) \rightarrow N \cap L^{2} \rightarrow 0$ is exact.
    \end{itemize}
   \end{prop} 
  
While Proposition \ref{a5} is very general and does not involve any additional property on the Lie algebras, the capability of certain families of nilpotent Lie algebras has been recently studied in \cite{johari1, johari2, N3}, because they play a fundamental role when we want to extend the arguments of classification via dimension.

From \cite{N3}, we recall that a finite dimensional Lie algebra $L$ is  \textit{Heisenberg} provided that $ L^{2}=Z(L) $ and dim $ L^{2}=1 $. Such algebras are odd dimensional with basis $ x_{1}, \cdots, x_{2m}, x$ and the only nonzero multiplication between basis elements is $x=[x_{2i-1}, x_{2i}]=- [x_{2i}, x_{2i-1}]$, for $i=1, \cdots, m$. The Heisenberg Lie algebra of dimension $2m+1$ is denoted by $H(m)$.

In order to look at Schur multipliers of these Lie algebras we recall a technical result, called \textit{K\"unneth Formula} (see \cite{weibel}). This result has categorical nature, but we report it in the way in which we will use it.

\begin{lem}[See Theorem 1 in \cite{es} and page 107 in \cite{ellis} ] \label{ab} 
For two Lie algebras $H$ and $K$ we have :
\begin{itemize}
\item[(i)]$(H \oplus K) \otimes (H \oplus K)\cong (H \otimes H) \oplus (K \otimes K) \oplus  (H \otimes K) \oplus  (K \otimes H),$
\item[(ii)]$(H \oplus K) \wedge (H \oplus K)\cong (H \wedge H) \oplus (K \wedge K) \oplus  (H/H^{2} \otimes K/K^{2}).$
\end{itemize}
   \end{lem}
   
It is also useful to recall here  Schur multipliers of $H(m)$ and $A(n)$.

    \begin{lem}[See  Corollary 2.3, \cite{N3}] \label{aa} \
     
      \begin{itemize}
    \item[(i)]$\mathrm{dim}~ M(A(n))= \frac{1}{2} n(n-1),$ 
 \item[(ii)] $\mathrm{dim}~M(H(1))= 2,$ 
\item[(iii)] $\mathrm{dim}~ M(H(m)) =2m^{2}-m-1$ for all $m \geq 2.$
       \end{itemize}
     \end{lem}
     
 Capable Lie algebras of dimension $n$ with $\mathrm{dim}~ L^{2}=1 $ are determined below and involve Heisenberg and abelian Lie algebras only. This is a first evidence of the role of Heisenberg algebras in the general theory.
  
  \begin{lem}[See Theorem 3.6, \cite{N3}]\label{z}
  Let $L$ be a nilpotent Lie algebra of dimension $n$ such that $\mathrm{dim}~ L^{2}=1 $. Then $L \cong H(m) \oplus A(n-2m-1)$ for some $m > 1$ and $L$ is capable if and only if $m=1$.
  \end{lem}

Removing the assumption that the derived subalgebra has dimension one in Lemma \ref{z}, we may ask whether we have  alternative conditions to determine capable Heisenberg algebras and capable abelian Lie algebras or not. The answer is positive and  reported below. 
 
  \begin{lem}[See Theorems 3.3, 3.4 and 3.5, \cite{N3}]\label{b2} 
$A(n)$ is capable if and only if $ n \geq 2$, and
$H(m)$ is capable if and only if $m=1$. Moreover,    for all $ k \geq 1$, the nilpotent Lie algebra $H(m) \oplus A(k)$ is capable if and only if $m=1$.
    \end{lem}

Recall from \cite{johari2} that  $L$ is a \textit{semidirect sum} of an ideal $I$ by a subalgebra $K$ if
$L = I + K$ and  $I  \cap K = 0$. We write $L=K \ltimes I$, when these circumstances occur. Of course, if both $K$ and $I$ are ideals of $L$ and $K \cap I=0$, then the decomposition (in direct sum) can be seen as a special case of the notion of semidirect sum.  Some important examples of semidirect sums are offered by the \textit{generalized Heisenberg algebras}; i.e.: a Lie algebra $L$ is called generalized Heisenberg of rank $n \ge 1$, if $H^2 = Z(H)$ and $\dim H^2 = n$. These have been studied  in \cite{johari2} and  allow us to  generalize  Lemma \ref{b2}.

Nevertheless the  relevance of Lemma \ref{b2} appears in the classification of finite dimensional nilpotent  Lie algebras via dimension only. In fact Lemma \ref{b2} offers a first elementary criterion to detect capable Lie algebras among those which are nilpotent and of finite dimension.  The classification  below  is reported from  \cite{serena, degraaf} and we enrich the statement with a description in terms of  semidirect sum.

\begin{thm}[Classification of  Nilpotent Lie Algebras of Dimension $\le$ 6]\label{classification}
Let $L$ be a finite dimensional nilpotent Lie algebra (over  $\mathbb{F}$ of characteristic $\neq 2$). Then
\begin{itemize}
\item[(1)] $\mathrm{dim} \ L =3$ if and only if $L$ is isomorphic to one of the following Lie algebras:
 \begin{itemize}
\item[-] $L_{3,1} = A(3)$,
\item[-] $L_{3,2}  = H(1) \simeq A(1) \ltimes A(2)$.
\end{itemize}

\item[(2)]$\mathrm{dim} \ L =4$ if and only if    $L$ is isomorphic to one of the following Lie algebras:
 \begin{itemize}
\item[-] $L_{4,1}  =   A(4)$,
\item[-] $L_{4,2}   = H(1) \oplus A(1),$
\item[-] $L_{4,3} = \langle x_1, x_2, x_3, x_4 \ | \ [x_1,x_2]=x_3, [x_1,x_3]=x_4 \rangle \simeq A(1) \ltimes A(3)$.
\end{itemize}

\item[(3)]$\mathrm{dim} \ L =5$ if and only if $L$ is isomorphic to one of the following Lie algebras:
\begin{itemize}
\item[-] $L_{5,1}  =  A(5) $,
\item[-] $L_{5,2}  = H(1) \oplus A(2)$,
\item[-] $L_{5,3} = L_{4,3} \oplus A(1) $,
\item[-] $L_{5,4} =   H(2) $,
\item[-] $L_{5,5} = \langle x_1, x_2, x_3, x_4, x_5 \ | \ [x_1,x_2]=x_3, [x_1,x_3]=[x_2,x_4]=x_5 \rangle \simeq  A(1) \ltimes L_{4,3},$
\item[-] $L_{5,6} =  \langle x_1, x_2, x_3, x_4, x_5 \ | \ [x_1,x_2]=x_3, [x_1,x_3]=x_4, [x_1,x_4]=[x_2,,x_3]=x_5  \rangle,$
\item[-] $L_{5,7} =  \langle x_1, x_2, x_3, x_4, x_5 \ | \ [x_1,x_2]=x_3, [x_1,x_3]=x_4, [x_1,x_4]=x_5 \rangle$,
\item[-] $L_{5,8} =  \langle x_1, x_2, x_3, x_4, x_5 \ | \ [x_1,x_2]=x_4, [x_1,x_3]=x_5 \rangle$,
\item[-] $L_{5,9} =  \langle x_1, x_2, x_3, x_4, x_5 \ | \ [x_1,x_2]=x_3, [x_1,x_3]=x_4, [x_2,x_3]=x_5 \rangle$.
\end{itemize}

\item[(4)]$\mathrm{dim} \ L =6$ if and only if $L$ is isomorphic to one of the following Lie algebras
\begin{itemize}
\item[-] $L_{6,k}=L_{5,k} \oplus A(1)$ for $k=1, 2, \ldots, 9$,
\item[-] $ L_{6,10}= \langle x_{1}, \cdots, x_{6}  \ | \   [x_{1},x_{2}]=x_{3},[x_{1},x_{3}]=x_{6}, [x_{4},x_{5}]=x_{6}\rangle, $
  \item[-] $ L_{6,11}= \langle x_{1}, \cdots, x_{6}  \ | \   [x_{1},x_{2}]=x_{3},[x_{1},x_{3}]=x_{4}, [x_{1},x_{4}]=x_{6}, [x_{2},x_{3}]=x_{6},[x_{2},x_{5}]=x_{6} \rangle \simeq  A(1) \ltimes L_{5,6},$
 \item[-] $ L_{6,12}= \langle x_{1}, \cdots, x_{6}  \ | \   [x_{1},x_{2}]=x_{3},[x_{1},x_{3}]=x_{4},[x_{1},x_{4}]=x_{6}, [x_{2},x_{5}]=x_{6} \rangle \simeq  A(1) \ltimes L_{5,7},$
   \item[-] $ L_{6,13}= \langle x_{1}, \cdots, x_{6}  \ | \   [x_{1},x_{2}]=x_{3},[x_{1},x_{3}]=x_{5}, [x_{2},x_{4}]=x_{5}, [x_{1},x_{5}]=x_{6},[x_{3},x_{4}]=x_{6} \rangle \simeq  A(1) \ltimes L_{5,7}, $
   \item[-] $ L_{6,14}= \langle x_{1}, \cdots, x_{6}  \ | \   [x_{1},x_{2}]=x_{3},[x_{1},x_{3}]=x_{4},  [x_{1},x_{4}]=x_{5}, [x_{2},x_{3}]=x_{5},[x_{2},x_{5}]=x_{6},[x_{3},x_{4}]=-x_{6} \rangle, $
    \item[-] $ L_{6,15}= \langle x_{1}, \cdots, x_{6}  \ | \   [x_{1},x_{2}]=x_{3},[x_{1},x_{3}]=x_{4}, [x_{1},x_{4}]=x_{5}, [x_{2},x_{3}]=x_{5},[x_{1},x_{5}]=x_{6},[x_{2},x_{4}]=x_{6}\rangle, $
   \item[-] $ L_{6,16}= \langle x_{1}, \cdots, x_{6}  \ | \   [x_{1},x_{2}]=x_{3},[x_{1},x_{3}]=x_{4}, [x_{1},x_{4}]=x_{5}, [x_{2},x_{5}]=x_{6},[x_{3},x_{4}]=-x_{6} \rangle,$
   \item[-] $ L_{6,17}= \langle x_{1}, \cdots, x_{6}  \ | \  [x_{1},x_{2}]=x_{3},[x_{1},x_{3}]=x_{4}, [x_{1},x_{4}]=x_{5}, [x_{1},x_{5}]=x_{6},[x_{2},x_{3}]=x_{6}\rangle, $
 \item[-]  $ L_{6,18}= \langle x_{1}, \cdots, x_{6}  \ | \   [x_{1},x_{2}]=x_{3},[x_{1},x_{3}]=x_{4}, [x_{1}, x_{4}]=x_{5}, [x_{1},x_{5}]=x_{6} \rangle,$
   \item[-] $ L_{6,19}(\varepsilon) =\langle x_{1}, \cdots, x_{6}  \ | \  [x_{1},x_{2}]=x_{4}, [x_{1},x_{3}]=x_{5}, [x_{1},x_{5}]=x_{6},[x_{2},x_{4}]=x_{6}, [x_{3},x_{5}]=\varepsilon x_{6}\rangle,$ 
    \item[-] $ L_{6,20}= \langle x_{1}, \cdots, x_{6}  \ | \   [x_{1},x_{2}]=x_{4},[x_{1},x_{3}]=x_{5}, [x_{1},x_{5}=x_{6}, [x_{2},x_{4}]=x_{6} \rangle ,$
  \item[-] $ L_{6,21}(\varepsilon)= \langle x_{1}, \cdots, x_{6}  \ | \   [x_{1},x_{2}]=x_{3},[x_{1},x_{3}]=x_{4}, [x_{2},x_{3}]=x_{5}, [x_{1},x_{4}]=x_{6},[x_{2},x_{5}]=\varepsilon x_{6} \rangle, $
    \item[-] $ L_{6,22}(\varepsilon) =\langle x_{1}, \cdots, x_{6}  \ | \   [x_{1},x_{2}]=x_{5},[x_{1},x_{3}]=x_{6}, [x_{2},x_{4}]=\varepsilon x_{6}, [x_{3},x_{4}]=x_{5} \rangle, $
  \item[-]  $ L_{6,23}= \langle x_{1}, \cdots, x_{6}  \ | \  [x_{1},x_{2}]=x_{3},[x_{1},x_{3}]=x_{5}, [x_{1},x_{4}]=x_{6},[x_{2},x_{4}]=x_{5} \rangle,$
 \item[-] $ L_{6,24}(\varepsilon)= \langle x_{1}, \cdots, x_{6}  \ | \   [x_{1},x_{2}]=x_{3},[x_{1},x_{3}]=x_{5}, [x_{1},x_{4}]=\varepsilon x_{6}, [x_{2},x_{3}]=x_{6},[x_{2},x_{4}]=x_{5} \rangle, $
   \item[-]$L_{6,25}=\langle x_{1}, \cdots, x_{6}  \ | \   [x_{1},x_{2}]=x_{3},[x_{1},x_{3}]=x_{5}, [x_{1},x_{4}]=x_{6} \rangle, $
    \item[-]$ L_{6,26}=\langle x_{1}, \cdots, x_{6}  \ | \   [x_{1},x_{2 }]=x_{4},[x_{1},x_{3}]=x_{5}, [x_{2},x_{3}]=x_{6} \rangle,$
   \item [-]  $L_{6,27}= \langle x_{1}, \cdots, x_{6} \ | \  [x_{1},x_{2}]=x_{3},[x_{1},x_{3}]=x_{5}, [x_{2},x_{4}]=x_{6} \rangle,$
 \item [-]  $ L_{6,28}= \langle x_{1}, \cdots, x_{6} \ | \  [x_{1},x_{2}]=x_{3},[x_{1},x_{3}]=x_{4}, [x_{1},x_{4}]=x_{5} , [x_{2},x_{3}]=x_{6} \rangle.$

  \end{itemize}
\end{itemize}
\end{thm}
Note that the presence of the parameter $\varepsilon$, which is a scalar of the ground field $\mathbb{F}$, determines $L_{6,19}(\varepsilon) \simeq  L_{6,19}(\delta)$ if and only if $\delta \varepsilon^{-1}$ is a square. Similarly,  
$L_{6,21}(\varepsilon) \simeq  L_{6,21}(\delta)$ if and only if $\delta \varepsilon^{-1}$ is a square, and the same condition gives criteria of isomorphisms for $L_{6,2}(\varepsilon) \simeq  L_{6,22}(\delta)$ and $L_{6,24}(\varepsilon) \simeq  L_{6,24}(\delta)$.

The main idea of the present paper is to describe the nonabelian exterior square and the capability of the above Lie algebras, but complications arise when the dimension is higher than six. 

\begin{rem}There are only finitely many ismorphism classes of nilpotent Lie algebras of dimension $n < 7$. In dimension $n=7$ there exist 1-parameter families of mutually non-isomorphic nilpotent Lie algebras. A classification has been achieved for $n=7$ over the real and complex numbers, by many different authors \cite{gong, seeley}. \end{rem}

The above remark shows our motivation to focus up to the case of dimension six. On the other hand, we will use a numerical argument for higher dimensions, loosing information in terms of generators and relations. Unfortunately, this difficulty is related to the nature of the topic. For sure, the abelian case  and that of Heisenberg algebras are clear by Lemmas \ref{aa}, \ref{b2} and  the following result.

\begin{lem}[See Lemma 3.2, \cite{N3} ]\label{a8}
For all $n \geq 1$, we have $A(n) \otimes A(n) \simeq A(n^{2})$ and $A(n) \wedge A(n) \simeq A(\frac{1}{2}n(n-1))$. For all $m \ge 2$, we have $H(m) \otimes H(m) \simeq A(4m^2)$ and $H(m) \wedge H(m) \simeq A(m(2m-1))$ 
\end{lem}

Because of Lemmas \ref{ab}, \ref{aa},  \ref{z},  \ref{b2}, \ref{a8}, the computations for Schur multipliers, nonabelian tensor squares and nonabelian exterior squares become significant from dimension $ \ge 4$. Moreover Lemmas \ref{b2} and \ref{a8} allow us to reduce to the nonabelian case of $\ge 4$, to the case of $\ge 4$ where no decompositions in Heisenberg and abelian Lie algebras are present. The next proof is emblematic.

\begin{lem}\label{b3}
We have 
\begin{itemize}
\item[(i)] $M(L_{4,1}) \simeq A(6)$, $M(L_{4,2}) \simeq A(4)$ and $M(L_{4,3}) \simeq A(2)$;
\item[(ii)]   $L_{4,1} \wedge L_{4,1} \simeq A(6)$,  $L_{4,2} \wedge L_{4,2} \simeq A(5)$, $L_{4,3} \wedge L_{4,3} \simeq A(4)$;
\item[(iii)] $L_{4,1} \otimes L_{4,1} \simeq A(16)$,  $L_{4,2} \otimes L_{4,2} \simeq A(11)$, $L_{4,3} \otimes L_{4,3} \simeq A(7)$.
\item[(iv)] $L_{4,1} \square  L_{4,1} \simeq A(10)$, $L_{4,2} \square  L_{4,2} \simeq A(6)$, $L_{4,3} \square  L_{4,3} \simeq A(3)$;
\end{itemize}
 \end{lem}

\begin{proof}(i). These Schur multipliers can be found in \cite{es}. (ii), (iii) and (iv) follow from
 Lemmas \ref{a8}, \ref{ab} and \cite[Corollary 2.14 (e)]{johari1}
 \end{proof}

Let's analyse the case of dimension five.

\begin{lem}\label{4} 
The Schur multipliers of the nilpotent Lie algebras $L_{5,k}$ (with $k=1, \cdots, 9 $) are the following:
\[M( L_{5,k}) \cong  \begin{cases}
A(10) & \textnormal {if $k=1$} , \\
A(7) & \textnormal {if $k=2$}, \\
A(4) & \textnormal {if $k=3, 5$}, \\
A(5) & \textnormal {if $k=4$}, \\
A(3) & \textnormal {if $k=6, 7, 9$}, \\
A(6) & \textnormal {if $k=8$}.
\end{cases} \]
 \end{lem}
 \begin{proof}
The case $ L_{5,1}$ follows by Lemma \ref{aa}. The case $L_{5,2}$ by Lemmas  \ref{ab} and  \ref{aa}.
Similarly we have the case of $L_{5,3}$. Now consider  $L_{5,4}$ and, instead of applying Lemma \ref{a8}, we use
 method of Hardy and Stitzinger in \cite{es}. We use 
\begin{align*}
[x_{1},x_{2}]=&x_{5} +s_{1},~~~~~~~~~~~~~~~[x_{1},x_{3}]=s_{2},~~~~~~~~~~~[x_{1},x_{4}]=s_{3},~~~~~~~~~~~~~~[x_{1},x_{5}]=s_{4}~~~~~~~~~~~~~~~~~
[x_{2},x_{3}]=s_{5},\\ [x_{2},x_{4}]=&s_{6},~~~~~~~~~~~[x_{2},x_{5}]=s_{7},~~~~~~~~~~~~~~~~~[x_{3},x_{4}]=x_{5} +s_{8},~~~~~~~~[x_{3},x_{5}]=s_{9},~~~~~~~~~~[x_{4},x_{5}]=s_{10},
\end{align*} 
Where $\lbrace s_{1},\cdots s_{10} \rbrace$ generate $M(L)$. A change of variables $s_{1}=s_{2}=0$. Next we use of the Jacobi identity on all possible triples. Shows
\begin{align*}
[x_{1},[ x_{2}, x_{3}]]+[x_{2}, [x_{3}, x_{1}]]+[x_{3},[ x_{1}, x_{2}]]=&0,\\
[x_{1},s_{6}]+[x_{2},s_{3}]+[x_{3}, x_{5}]=&0,\\
s_{9}=0.
\end{align*} 
By a similar technique, we have $s_{8}=s_{10}=0$. Thus $\mathrm{dim}~M(L)=5$.

 By the same argument, we can see that
  $\mathrm{dim}~M(L_{5,5})=4$, $\mathrm{dim}~M(L_{5,6})=3$, $\mathrm{dim}~M(L_{5,7})=3$, $\mathrm{dim}~M(L_{5,8})=6$ and $\mathrm{dim}~M(L_{5,9})=3$.
 \end{proof}

The following lemma summarises some well known facts when $1 \le k \le 5 $, but there are additional information for $6 \le k \le 9$.

 \begin{lem}\label{a10}\
For Lie algebras $L_{5,k}$ and positive integers $k= 1,\cdots, 9$ we have
 
\[L_{5,k} \wedge L_{5,k} \cong  \begin{cases}
A(6) & \textnormal {if $k=3,4,5,9$} , \\
A(8) & \textnormal {if $k=2,8$}, \\
A(10) & \textnormal {if $k=1$}, \\
H(1) \oplus A(3) & \textnormal {if $k=6,7$}. \\

\end{cases} \]
 \end{lem}
 \begin{proof}
By Lemma \ref{4}, we have  $L_{5,1} \wedge L_{5,1} \cong A(10)$. Lemmas \ref{ab} and \ref{b3} imply 
 $$ (L_{5, 2} \wedge L_{5,2}) = (L_{4,2} \wedge L_{4,2}) \oplus  (A(1) \wedge A(1)) \oplus ( L_{4,2}/ L_{4,2} ^{2} \otimes A(1)) \simeq A(8).$$ Similarly $ \mathrm{dim}~(L_{5, 3} \wedge L_{5,3})= 6$ and it is an abelian Lie algebra.
 Now cosider $L_{5,4}$. It is clear that $(L \wedge L)^{2} \leq L^{2} \wedge L^{2}$.
  Since $L_{5,4}^{2}\ \simeq A(1)$, $ \dim (L_{5,4}\wedge L_{5,4})^{2}= \dim (L_{5,4}^{2}\wedge L_{5,4}^{2})=0$ by Lemma \ref{aa}. Hence $ L_{5,4}\wedge L_{5,4} $ is abelian.
 By using the fact that $$ \mathrm{dim}~ L_{5,4} \wedge L_{5,4}=\mathrm{dim}~M(L_{5,4}) + \mathrm{dim}~L_{5,4}^{2}$$ and  Lemma \ref{4}, we have $L_{5,4} \wedge L_{5,4} \cong A(6)$. The case of $L \cong L_{5,5}$ either can be found directly or refer to \cite[Corollary 2.15 (f)]{johari1}.
 Let $L_{5,6}=\langle x_{1}, \dots, x_{5} \mid [x_{1},x_{2}]=x_{3},[x_{1},x_{3}]=x_{4}, [x_{1},x_{4}]=x_{5}, [x_{2},x_{3}]=x_{5}\rangle$ by using  \cite[$(6^{\prime})$]{ellis2} and the relations of $ L $, we have
 \begin{align*}
 & x_{1}\wedge x_{5}=-([x_{2}, x_{3}]\wedge x_{1})=-(x_{2}\wedge [x_{3}, x_{1}]- x_{3}\wedge [x_{2}, x_{1}])=x_{2}\wedge x_{4},\cr
 & x_{2}\wedge x_{5}=-([x_{1}, x_{4}]\wedge x_{2})=-(x_{1}\wedge [x_{4}, x_{2}]-x_{4}\wedge[x_{1}, x_{2}])=x_{4}\wedge x_{3},\cr
 &x_{3}\wedge x_{5}=x_{4}\wedge x_{5}=0.
 \end{align*}
 Therefore $ \lbrace  x_{1}\wedge x_{2}, x_{1}\wedge x_{3}, x_{1}\wedge x_{4},  x_{1}\wedge x_{5}, x_{2}\wedge x_{3}, x_{2}\wedge x_{5} \rbrace$ is a generating set for $ L\wedge L $. 
 Now we obtain the relations of $ L\wedge L $. By using \cite[(7)]{ellis2} and the relations of $ L $, we have
 \begin{align*}
 &[x_{1}\wedge x_{2}, x_{1}\wedge x_{3}]=[x_{1}, x_{2}]\wedge [x_{1}, x_{3}]=x_{3}\wedge x_{4}=-x_{2}\wedge x_{5}\cr
 & [x_{1}\wedge x_{2}, x_{1}\wedge x_{4}]=[x_{1}, x_{2}]\wedge [x_{1}, x_{4}]=x_{3}\wedge x_{5}=0\cr
 & [x_{1}\wedge x_{2}, x_{2}\wedge x_{3}]=[x_{1}, x_{2}]\wedge [x_{2}, x_{3}]=x_{3}\wedge x_{5}=0\cr
 &[x_{1}\wedge x_{3}, x_{2}\wedge x_{4}]=[x_{1}, x_{3}]\wedge [x_{2}, x_{4}]=x_{4}\wedge x_{5}=0\cr
 & [x_{1}\wedge x_{3}, x_{2}\wedge x_{3}]=[x_{1}, x_{3}]\wedge [x_{2}, x_{3}]=x_{4}\wedge x_{5}=0\cr
 & [x_{1}\wedge x_{4}, x_{2}\wedge x_{3}]=[x_{1}, x_{4}]\wedge [x_{2}, x_{3}]=x_{5}\wedge x_{5}=0
 \end{align*}
  Hence 
  $$L_{5,6}\wedge L_{5,6} \cong H(1) \oplus A(3).$$
 By the same argument, we have $L_{5, 7} \wedge L_{5,7} \cong H(1)\oplus A(3)$,  $L_{5, 8} \wedge L_{5,8} \cong A(8)$ and $L_{5, 9} \wedge L_{5,9} \cong A(6)$. 

\end{proof}

\begin{lem}\label{lem1}
For Lie algebras $L_{5,k}$ and positive integers $k= 1,\cdots, 9$ we have
\[L_{5,k} \square L_{5,k} \cong  \begin{cases}
A(3) & \textnormal {if $k=6,7,9$} , \\
A(6) & \textnormal {if $k=3,5,8$}, \\
A(10) & \textnormal {if $k=2,4$}, \\
 A(15) & \textnormal {if $k=1$}. \\
\end{cases} \]
\begin{proof}
Let $ L $ be a $ n $-dimensional nilpotent Lie algebra  with the derived subalgebra of dimention $ m $. Then $ L\square L\cong L/L^{2}\square L/L^{2} $ and $ \dim L/L^{2}\square L/L^{2}=\frac{1}{2}(n-m)(n-m+1)$ by using \cite[Lemma 2.1 and 2.2]{N1}. Now by inserting values $ n $ and $ m $ for all $ L_{5,k} $ such that $ k=1,\dots,9 $ the result follows.
\end{proof}
\end{lem}
\begin{cor}\label{conseq2}
For Lie algebras $L_{5,k}$ and positive integers $k= 1,\cdots, 9$ we have

\[L_{5,k} \otimes L_{5,k} \cong  \begin{cases}
A(9) & \textnormal {if $k=9$} , \\
A(12) & \textnormal {if $k=3,5$}, \\
A(14) & \textnormal {if $k=8$}, \\
A(16) & \textnormal {if $k=4$}, \\
A(18) & \textnormal {if $k=2$}, \\
A(25) & \textnormal {if $k=1$},\\
H(1) \oplus A(6) & \textnormal {if $k=6,7$}.\\
\end{cases} \]
\begin{proof}
By  \cite[Theorem 2.5]{N1}, we have $ L\otimes L\cong L\wedge L\oplus L\square L. $
 Now by using Lemmas \ref{a10} and  \ref{lem1}   the result follows.
\end{proof}
\end{cor}

\section{Proofs of the main results}

We begin to describe directly the Schur multipliers of  $L_{6,k}$ for $k = 1, \ldots, 28$.

\begin{prop}\label{5}The Schur multiplier of 6-dimensional nilpotent Lie algebras are abelian algebras $L_{6,k}$ for $k=1, \cdots, 28$. In particular,
 \[M( L_{6,k}) \cong  \begin{cases}
 A(2) & \textnormal {if $k=14,16$} , \\
 A(3) & \textnormal {if $k=15,17,18$}, \\
 A(4) &
 \textnormal {if $k=13,21(\forall \varepsilon),28$},\\
A(5) & \textnormal {if $k=6,7,9,11,12,19(\forall \varepsilon),20,24(\forall \varepsilon)$} , \\
A(6) &
 \textnormal {if $k=10,23,25,27$}, \\
A(7) & \textnormal {if $k=3,5$}, \\
A(8) & \textnormal {if $k= 22(\forall \varepsilon),26$},\\
A(9) & \textnormal {if $k=4,8$}, \\
A(11) &
 \textnormal {if $k=2$},\\
A(15) & \textnormal {if $k=1$}.\\
\end{cases} \].
\end{prop}

\begin{proof}
By Lemma \ref{aa}, we have $\mathrm{dim}~M(L_{6,1})=15$. By Lemmas \ref{ab} and \ref{4}, the dimension of $M(L_{6,k})$ is obtained for $k=1, \cdots,9$. Now consider $ L \cong L_{6,10}$ and apply the argument we have used in Lemma \ref{a10}. Start with
\begin{align*} [x_{1},x_{2}]=&x_{3}+s_{1},~~~~~~~~~~~~~[x_{1},x_{3}]=x_{6}+s_{2},~~~~~~~~~~~[x_{1},x_{4}]=s_{3},~~~~~~~~~~~~~~~~[x_{1},x_{5}]=s_{4},~~~~~~~~~~~
[x_{1},x_{6}]=s_{5}\\
 [x_{2},x_{3}]=&s_{6},~~~~~~~~~~[x_{2},x_{4}]=s_{7},~~~~~~~~~~~~[x_{2},x_{5}]=s_{8},~~~~~~~~~~~[x_{2},x_{6}]=s_{9},~~~~~~~~~~~~~[x_{3},x_{4}]=s_{10}\\
 [x_{3},x_{5}]=&s_{11},~~~~~~~~~~~[x_{3},x_{6}]=s_{12},~~~~~~~~~~~~[x_{4},x_{5}]=x_{6}+s_{13},~~~~~~~~~~~~~~~~[x_{4},x_{6}]=s_{14}~~~~~~~~~~~~[x_{5},x_{6}]=s_{15},
 \end{align*}
where $\lbrace s_{1},\cdots s_{15} \rbrace$ generated $M(L)$. A change of variables allows  that $s_{1}=s_{2}=0$. By use of the Jacobi identity on all possible triples, we have $s_{9}=s_{10}=s_{11}=s_{12}=s_{14}=s_{15}=0$.
Thus $\mathrm{dim}~M(L)=6$.

Now let $L \cong L_{6,11}$. Start with
\begin{align*} [x_{1},x_{2}]=&x_{3}+s_{1},~~~~~~~~~~~~~[x_{1},x_{3}]=x_{4}+s_{2},~~~~~~~~~~~[x_{1},x_{4}]=x_{6}+s_{3},~~~~~~~~~~~~~~~~[x_{1},x_{5}]=s_{4},~~~~~~~~~~~
[x_{1},x_{6}]=s_{5},\\
 [x_{2},x_{3}]=&x_{6}+s_{6},~~~~~~~~~~[x_{2},x_{4}]=s_{7},~~~~~~~~~~~~[x_{2},x_{5}]=x_{6}+s_{8},~~~~~~~~~~~[x_{2},x_{6}]=s_{9},~~~~~~~~~~~~~[x_{3},x_{4}]=s_{10},\\
 [x_{3},x_{5}]=&s_{11},~~~~~~~~~~~[x_{3},x_{6}]=s_{12},~~~~~~~~~~~~[x_{4},x_{5}]=x_{6}+s_{13}~~~~~~~~~~~~~~~~[x_{4},x_{6}]=s_{14}~~~~~~~~~~~~[x_{5},x_{6}]=s_{15},
\end{align*}
where $\lbrace s_{1},\cdots s_{15} \rbrace$ generate $M(L)$. 

A change of variables shows that $s_{1}=s_{2}=s_{3}=0$. Use of the Jacobi identity on all possible triples, we have
 \begin{align*}
 [x_{1},[ x_{2}, x_{3}]]+[x_{2},[ x_{3}, x_{1}]]+[x_{3}, [x_{1}, x_{2}]]=&0\\
 [x_{1}, x_{6}] - [x_{2}, x_{4}]= &0.
 \end{align*}
 Thus $s_{5}=s_{7}$. Again use of the Jacobi identity, we have
  \begin{align*}
 [x_{1},[ x_{2}, x_{4}]]+[x_{2},[ x_{4}, x_{1}]]+[x_{4}, [x_{1}, x_{2}]]=&0\\
 -[x_{2}, x_{6}] - [x_{4}, x_{3}]= &0.
 \end{align*}
 Therefore $s_{9}=-s_{10}$.
 By a same way we have $s_{11}=s_{12}=s_{13}=s_{14}=s_{15}=0$. Thus $\mathrm{dim}~M(L)=5$.
 The rest of proof is obtained by a similar technique, using the same idea (or via GAP \cite{gap}) when $12 \le k \le 28$.
\end{proof}

Now we describe  the nonabelian exterior square up to dimension $6$.

 \begin{prop}\label{a11}\
For Lie algebras $L_{6,k}$ with $k= 1,\ldots, 28$, we have
\[L_{6,k} \wedge L_{6,k}  \cong  \begin{cases}
 A(7) & \textnormal {if $k=13$},\\
 A(8) & \textnormal {if $k=9,10,19(\forall \varepsilon),20,24(\forall \varepsilon)$} , \\
 A(9) & \textnormal {if $k=3,5,23,25,27$}, \\
 A(10) & \textnormal {if $k=4,22(\forall \varepsilon)$},\\
A(11) & \textnormal {if $k=8,26$} , \\
A(12) & \textnormal {if $k=2$}, \\
A(15) & \textnormal {if $k=1$}, \\
H(1) \oplus A(3)& \textnormal {if $k= 16$},\\
H(1) \oplus A(4) & \textnormal {if $k=15,17,18$}, \\
H(1) \oplus A(5) & \textnormal {if $k=6,7,11,12,28,21(\varepsilon =0) $},\\
L_{5,8} \oplus A(1)  & \textnormal {if $k=14$},\\
L_{5,8} \oplus A(3)  & \textnormal {if $k=21(\varepsilon \neq 0)$},\\
\end{cases} \]
\end{prop}

\begin{proof} By Lemmas \ref{ab} and \ref{a10}, the result is clear for $k=1,\ldots, 9$. Let $ L \cong L_{6,10}=\langle x_{1}, \dots, x_{6}\mid [x_{1}, x_{2}]=x_{3}, [x_{1}, x_{3}]=x_{6}, [x_{4}, x_{5}]=x_{6}\rangle. $ By using \cite[$(6^{\prime})$]{ellis2} and the relations of $ L $, we have 
\begin{align*}
& x_{1}\wedge x_{6}=x_{1}\wedge [x_{4}, x_{5}]=-(x_{4}\wedge [x_{5}, x_{1}]-x_{5}\wedge [x_{4}, x_{1}])=0\cr
&  x_{2}\wedge x_{6}=x_{2}\wedge [x_{1}, x_{3}]=-(x_{1}\wedge [x_{3}, x_{2}]-x_{3}\wedge [x_{1}, x_{2}])=0\cr
&x_{3}\wedge x_{4}=x_{3}\wedge x_{5}=x_{3}\wedge x_{6}=x_{4}\wedge x_{6}=x_{5}\wedge x_{6}=0
\end{align*}
Therefore $ \lbrace  x_{1}\wedge x_{2}, x_{1}\wedge x_{3}, x_{1}\wedge x_{4},  x_{1}\wedge x_{5}, x_{2}\wedge x_{3}, x_{2}\wedge x_{4}, x_{2}\wedge x_{5}, x_{4}\wedge x_{5}\rbrace$ is a generating set for $ L\wedge L $. 
Now we obtain the relations of $ L\wedge L $. By using \cite[(7)]{ellis2} and the relations of $ L $, we have
\begin{align*}
 &[x_{1}\wedge x_{2}, x_{1}\wedge x_{3}]=[x_{1}, x_{2}]\wedge [x_{1}, x_{3}]=x_{3}\wedge x_{6}=0\cr
 & [x_{1}\wedge x_{2}, x_{1}\wedge x_{4}]=[x_{1}, x_{2}]\wedge [x_{4}, x_{5}]=x_{3}\wedge x_{6}=0\cr
 & [x_{1}\wedge x_{3}, x_{4}\wedge x_{5}]=[x_{1}, x_{3}]\wedge [x_{4}, x_{5}]=x_{6}\wedge x_{6}=0
 \end{align*}
 Thus $ L\wedge L\cong A(8). $ In a similar way, we  obtain the exterior square of the remaining Lie algebras of dimension $ 6 $.
\end{proof}

We have all that we need for the proof of the first main theorem.

  \begin{proof}[Proof of Theorem \ref{main1}] A first distinction comes from the abelian case and from the nonabelian case, looking at Theorem \ref{classification}. Clearly Lemma \ref{b2} shows that  $A(1)$ is noncapable and $ A(n) $ for $ n\geq 2 $ is capable, so, up to dimension $\le 6$ the only abelian noncapable Lie algebra is $A(1)$: this means that there is no loss of generality in  assuming in the rest of the proof that the Lie algebras are nonabelian. We look again at Theorem \ref{classification} in case of dimension $3$ and find  that there are no noncapable Lie algebras by Lemma \ref{b2}. The same happens also in dimension $4$ looking at Theorem \ref{classification}: indeed  $L_{4,2}$ is capable by Lemma \ref{b2}.  By Proposition \ref{a5} (ii) and (iv), we can see that for any central ideal $K$ of dimension one in a finite dimensional nilpotent Lie algebra $L$,   
  $$ \dim M(L)=\dim M(L/K)-\dim (L^{2}\cap K) \ \ \Longleftrightarrow  \ \  K \subseteq Z^\wedge(L). \leqno{(\dag)}$$ 
  Let's apply this when $ L \simeq L_{4,3}$, in order to show that $Z^\wedge(L)=0$.  By  Lemma \ref{b3}, we have  $  \dim M(L_{4,3 })=2$ . Note that $ Z(L_{4,3})=\langle x_{4}\rangle $, so  either $ Z^{\wedge}(L_{4,3})=0 $,  or  $Z^{\wedge}(L_{4,3})= Z(L_{4,3}). $ In the first case, the claim follows. In the second case $ L_{4,3}/Z(L_{4,3})\simeq L_{3,2} $ and $ \dim M(L_{3,2})=2 $ by Lemma \ref{aa} (i). Therefore $(\dag)$ implies $$ \dim M(L_{4,3})\neq \dim M(L_{4,3}/Z(L_{4,3})-\dim (L_{4,3}^{2} \cap Z(L_{4,3}))$$  and so  $Z(L_{4,3})=Z^\wedge(L_{4,3})$ cannot happen. We conclude that $L_{4,3}$ is capable and that there are no noncapable Lie algebras of dimension $4$.

  Now we consider Lie algebras of the dimension $5$ and we see that here we have the first examples of noncapable Lie algebras. Of course, $ L_{5,2}$ is capable by Lemma \ref{b2}. Therefore we consider $ L_{5,3}$ with $K$ any  ideal of dimension one in $Z(L_{5,3}) \simeq A(2)$. Depending on the possible two choices of $K$ in $Z(L_{5,3})$,  the quotient $ L_{5,3}/K $ is isomorphic either to $ L_{4,1}, $ or to $ L_{4,2}$. Also $ \dim M(L_{5,3})=4 ,$ $ \dim M(L_{4,1})=6 $ and $ \dim M(L_{4,2})=4 $ by  Lemmas \ref{b3} and \ref{4}.  Again we apply $(\dag)$ and note that 
  $$ \dim M(L_{5,3})\neq \dim M(L_{5,3}/K)-\dim (L_{5,3}^{2}\cap K)$$ for all one dimensional central ideals of $ L_{5,3} $, so we have necessarily $ Z^{\wedge}(L_{5,3})=0$, hence $L_{5,3}$ is capable.

Then we pass to consider $L_{5,4}$ and clearly this is noncapable by Lemma \ref{b2}. 
Now we examine the case of  $L_{5,k}$, when $k = 5,\ldots, 9$ and these are capable Lie algebras, because we can apply the same argument which has been applied for  $L_{5,3}$ with help of Lemma \ref{4}. The same idea applies in the case of dimension six and we find  that $L_{6,k}$ (for $k =4,  10, 14, 16, 19, 20$) are the only noncapable Lie algebras. Here  Proposition \ref{5} is used. Then the result follows.

 \end{proof}

Theorem \ref{main1} shows an evidence, which we report below.

\begin{prop}\label{francesco1} Any finite dimensional nonabelian nilpotent Lie algebra  of dimension  $ \le 6$ has   capable nonabelian exterior square.
\end{prop}

\begin{proof}We begin to show the thesis up to dimension $\le 5$.  If $L$ has $\mathrm{dim} \ L \le 5$, then Theorem \ref{classification} shows that  $\mathrm{dim} \ L^2  \le 3$, so three options are possible. The first is that $\mathrm{dim} \ L^2  =1$, hence \cite[Theorem 2.15 (i)]{johari1} and Lemma \ref{b2} show that $L \wedge L$ is always capable, because we get an abelian Lie algebra $A(n)$ of dimension $n$ always bigger than $2$. The second option is that   $\mathrm{dim} \ L^2 =2$, hence \cite[Theorem 2.15 (ii) and (iii)]{johari1} and Lemma \ref{b2} show   that $L \wedge L$ is always capable, because we get an abelian Lie algebra $A(n)$ of dimension always bigger than $4$. Finally, one can see that the five dimensional nonabelian nilpotent Lie algebras with derived subalgebra of dimension 3 can be the remaining $L_{5,6}$, $L_{5,7}$, $L_{5,8}$ and $L_{5,9}$ and here the result follows by Lemma \ref{a10}.

Assume now that $\mathrm{dim} \ L \le 6$ and consider Lemma \ref{b2} and Proposition \ref{a11}. The only two Lie algebras, which we must check, are $L_{6,14} \wedge L_{6,14}$ and $L_{6,21} \wedge L_{6,21}$. First of all we look at generators and relations in Theorem \ref{classification} and recognize easily that $L_{5,8}$ is generalized Heisenberg of rank two. Then we invoke \cite[Proposition 2.4]{johari2} which shows  that the capability of $L_{5,8} \oplus A(1)$ is equivalent to that of the factor which is generalized Heisenberg. Since $L_{5,8}$  is capable, we may conclude that $L_{5,8} \oplus A(1)$, hence $L_{6,14} \wedge L_{6,14}$, is capable. The same argument applies to $L_{6,21} \wedge L_{6,21}$ and the result follows.
\end{proof}

Noting that the derived subalgebra of a Lie algebra of dimension $\le 5$ is always capable,  we are motivated to formulate  
 the following question.

\begin{que}\label{oq} When can we detect the capability of $L \wedge L$ from that  of $L^2/Z^\wedge(L)$ ? In other words, it is interesting to decide when  the following implication  is true:
$$L^2/Z^\wedge(L) \  \mbox{nilpotent capable Lie algebra} \ \Rightarrow \ L \wedge L \ \mbox{nilpotent capable Lie algebra}$$
\end{que}

\medskip
\medskip

There are  important evidences which motivate the above question.

\begin{ex}
Consider the Lie algebra   $L \simeq H(2) \oplus A(1)$. Here $L^2 \simeq [H(2),H(2)] \simeq A(1)$ which is noncapable by Lemma \ref{b2}, but $L \wedge L \simeq A(k)$ with $k \ge 2$ by Lemma \ref{ab} (ii) and so  this Lie algebra is capable. Note that $L^2/Z^\wedge(L)$ is the trivial Lie algebra, which is of course capable. Therefore  $L^2$ noncapable may imply that $L \wedge L$ is capable, but definitively this example supports  Question \ref{oq} positively. Another evidence is offered by the generalized Heisenberg algebras, studied in \cite{johari1, johari2, johari3, johari4}. If $K$ is a $d$-generated generalized Heisenberg algebra of rank $\frac{1}{2} d(d-1) $, then $K$ is capable by \cite[Theorem 2.7]{johari4}. In particular, $K^2 \simeq A(\frac{1}{2} d(d-1))$ so it is always capable when $d \ge 3$. This means that $K^2/Z^\wedge(K)$ is capable and in fact $K \wedge K \simeq A(k)$ with $k \ge2$ by an an application of  \cite[Theorem 2.8]{johari4}.
\end{ex}

The proof of our second main result answers Question \ref{oq} partially.

\begin{proof}[Proof of Theorem \ref{francesco2}]Since  $L$ is nonabelian and nilpotent,  $L \neq L^{2}$ and $\mathrm{dim} \ L^2 \ge 1$ so that $\mathrm{dim} \ L/L^2 \ge 2$. This implies $Z^\wedge(L/L^2)=0$, because $L/L^2$ is an abelian capable Lie algebra by Lemma \ref{b2} (i), and so $Z^\wedge(L) \subseteq L^2$ concluding that $Z^\wedge(L)$ is an ideal of $L^2$, that is, the Lie algebra quotient $L^2/Z^\wedge(L)$ is well defined.

   Proposition \ref{a5} (ii) shows that the map $$\pi: x \wedge y \in L \wedge L \mapsto (x\wedge y) + Z^\wedge(L) \in L/Z^\wedge(L) \wedge L/Z^\wedge(L)$$ is an isomorphism of Lie algebras, so it will be sufficient to show that $L/Z^\wedge(L) \wedge L/Z^\wedge(L)$ is capable, in order to conclude that so is $L \wedge L$.

 Note that   the central extension
$$0 \longrightarrow M\left(\frac{L}{Z^\wedge(L)}\right)  \longrightarrow \frac{L}{Z^\wedge(L)} \wedge \frac{L}{Z^\wedge(L)}   \longrightarrow \left[\frac{L}{Z^\wedge(L)},\frac{L}{Z^\wedge(L)}\right] \longrightarrow 0 $$  implies $$\frac{L/Z^\wedge(L) \wedge L/Z^\wedge(L)}{M(L/Z^\wedge(L))} \simeq  \left[\frac{L}{Z^\wedge(L)},\frac{L}{Z^\wedge(L)}\right] = \frac{L^2+Z^\wedge(L)}{Z^\wedge(L)} \simeq  \frac{L^2}{Z^\wedge(L)} $$
$$\Rightarrow \  Z^\wedge\left(\frac{L/Z^\wedge(L) \wedge L/Z^\wedge(L)}{M(L/Z^\wedge(L))}\right) \simeq  Z^\wedge\left(\frac{L^2}{ Z^\wedge(L)}\right)  =0,$$
where the last condition is possible because  $L^2/  Z^\wedge(L)$ is capable.

Since the following inclusion is true: 
$$ \frac{Z^\wedge (L/Z^\wedge(L) \wedge L/Z^\wedge(L))}{M(L/Z^\wedge(L))} \subseteq  Z^\wedge\left(\frac{L/Z^\wedge(L) \wedge L/Z^\wedge(L)}{M(L/Z^\wedge(L))}\right),$$
we get  $$Z^\wedge(L/Z^\wedge(L) \wedge L/Z^\wedge(L)) \subseteq  M(L/Z^\wedge(L))$$ and the result follows.
\end{proof}

\end{document}